\documentclass[reqno]{amsart}
\usepackage{amssymb, amsmath, amsthm, amsfonts,graphicx}

\usepackage{mathrsfs}

\def\G{{\mathscr G}}
\def\C{{\mathcal C}}

\theoremstyle{plain}
\newtheorem{theorem}{Theorem}

\newtheorem{lemma}{Lemma}

\theoremstyle{definition}
\newtheorem{definition}{Definition}

\numberwithin{equation}{section}
\numberwithin{theorem}{section}
\numberwithin{definition}{section}
\numberwithin{lemma}{section}
\numberwithin{corollary}{section}
\numberwithin{prop}{section}
\numberwithin{remark}{section}
\numberwithin{example}{section}

\newcommand{\s}{\sigma}
\newcommand{\m}{\mu}
\renewcommand{\t}{\tau}
\newcommand{\D}{\Delta}

\newcommand{\R}{\ensuremath{\mathbb{R}}}
\newcommand{\Z}{\ensuremath{\mathbb{Z}}}
\def\H{{\mathcal H}}
\newcommand{\N}{\ensuremath{\mathbb{N}}}

\newcommand{\T}{\ensuremath{\mathbb{T}}}
\newcommand{\I}{\ensuremath{\mathbb{I}}}

\DeclareMathOperator{\myRe}{Re}
\DeclareMathOperator{\myIm}{Im}

\def\Rem{{\myRe}_\mu}
\def\Imm{{\myIm}_\mu}
\def\({\left(}
\def\){\right)}

\DeclareMathOperator{\Log}{Log}

\usepackage{array}

\begin{document}


\title[Linear State Feedback Stabilization on Time Scales]{Linear State Feedback Stabilization\\on Time Scales}
\author[Jackson, Davis, Gravagne, Marks]{Billy J. Jackson$^1$, John M. Davis$^2$, Ian A. Gravagne$^3$, Robert J. Marks II$^3$}
\address{$^1$School of Mathematical Sciences, University of Northern Colorado, Greeley, CO 80639}
\email{billy.jackson@unco.edu}
\address{$^2$Department of Mathematics, Baylor University, Waco, TX 76798}
\email{John\_M\_Davis@baylor.edu}
\address{$^3$Department of Electrical and Computer Engineering, Baylor University, Waco, TX 76798}
\email{Ian\_Gravagne@baylor.edu,  Robert\_Marks@baylor.edu}
\thanks{*This work was supported by NSF Grant CMMI\#726996. Please see {\tt http://www.timescales.org/} for other papers from the Baylor Time Scales Research Group.}

\begin{abstract}
For a general class of dynamical systems (of which the canonical continuous and uniform discrete versions are but special cases), we prove that there is a state feedback gain such that the resulting closed-loop system is uniformly exponentially stable with a prescribed rate. The methods here generalize and extend Gramian-based linear state feedback control to much more general time domains, e.g. nonuniform discrete or a combination of continuous and discrete time. In conclusion, we discuss an experimental implementation of this theory.
\end{abstract}

\maketitle


\section{Introduction}

Linear systems theory is well-studied in both the continuous and discrete settings \cite{AnMi,CaDe,Ru}, but recently an important line of investigation has been generalizing the known linear systems theory on $\R$ and $\Z$ to nonuniform discrete domains or domains with a mixture of discrete and continuous parts. Progress toward this has been made on the topics of controllability/observability and reachability/realizability \cite{DaGrJaMa, Ja}, Laplace transforms \cite{DaGrJaMaRa,DaGrMa1, DaGrMa2}, Fourier transforms \cite{MaGrDa}, Lyapunov equations \cite{Da3, DaGrMaRa}, and various types of stability results including Lyapunov, exponential, and BIBO \cite{BoMa,Da3,Ja}. The goal is not to simply reprove existing, well-known theories, but rather to view $\R$ and $\Z$ as special cases of a single, overarching theory and to extend the theory to dynamical and control systems on these more general domains. Doing so reveals a rich mathematical structure which has great potential for new applications in diverse areas such as adaptive control \cite{GrDaDa}, real-time communications networks \cite{GrDaDaMa,GrDaMa}, dynamic programming \cite{SeSaWu}, switched systems \cite{MaGrDaDa}, stochastic models \cite{BhEvPeRa}, population models \cite{Zh}, and economics \cite{AtBiLe,AtUy}. The focus of this paper is the study of linear state feedback controllers \cite{Ka, SuFo} in this generalized setting and to compare and contrast these results with the standard continuous and uniform discrete scenarios.

\section{Time Scales Background}

\subsection{What Are Time Scales?}

The theory of time scales springs from the 1988 doctoral dissertation of Stefan Hilger \cite{Hi2} that resulted in his seminal paper \cite{Hi1}. These works aimed to unify various overarching concepts from the (sometimes disparate) theories of discrete and continuous dynamical systems \cite{MiHoLi}, but also to extend these theories to more general classes of dynamical systems. From there, time scales theory advanced fairly quickly, culminating in the excellent introductory text by Bohner and Peterson \cite{BoPe2} and the more advanced monograph \cite{BoPe1}. A succinct survey on time scales can be found in \cite{AgBoORPe}.

\begin{table}
\caption{Canonical time scales compared to the general case.}
\label{comparisons}
\renewcommand{\arraystretch}{1.25}
\begin{tabular}{|>{\centering}m{1.75cm}|>{\centering}m{1.5in}|>{\centering}%
  m{1.5in}|>{\centering\arraybackslash}m{1.5in}|}
  \hline
   & \sf{continuous} & \sf{(uniform) discrete} & \sf{time scale}\\ \hline
  \sf{domain} & $\R$ & $\Z$ & $\T$ \\ \hline
  & \rule{0ex}{3em}\includegraphics[scale=.35]{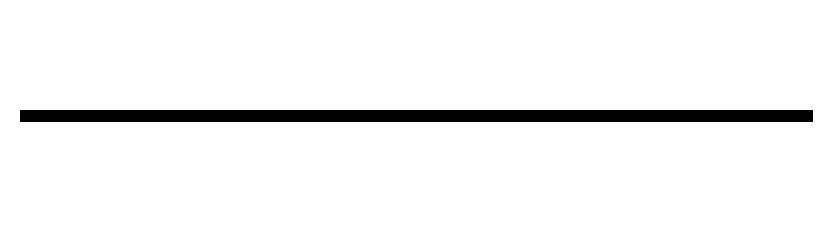} & \includegraphics[scale=.35]{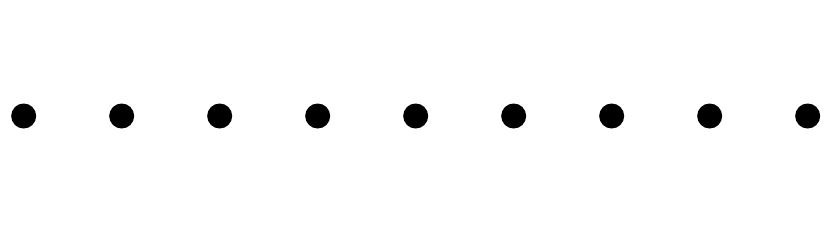} & \includegraphics[scale=.35]{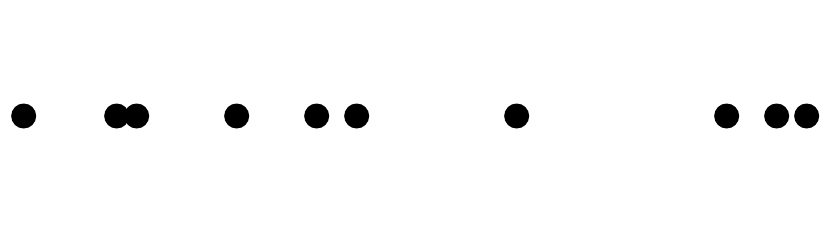}\\ \hline
  \sf{forward jump} & $\sigma(t)\equiv t$ & $\sigma(t)\equiv t+1$ & $\sigma(t)$ varies\\ \hline
  \sf{step size} & $\mu(t)\equiv 0$ & $\mu(t)\equiv 1$ & $\mu(t)$ varies\\ \hline
  \sf{differential operator} & \footnotesize{$\displaystyle\dot{x}(t):=\lim_{h\to 0}{x(t+h)-x(t)\over h}$} & \footnotesize{$\Delta x(t):=x(t+1)-x(t)$} & \footnotesize{$\displaystyle{x^\Delta(t):={x(t+\mu(t))-x(t)\over \mu(t)}}$}\\ \hline
  \sf{canonical equation} & $\dot{x}(t)=Ax(t)$ & $\Delta x(t)=Ax(t)$ & $x^\Delta(t)=Ax(t)$\\ \hline
  \sf{LTI stability region in $\mathbb C$} & \rule{0ex}{9em}\includegraphics[scale=.35]{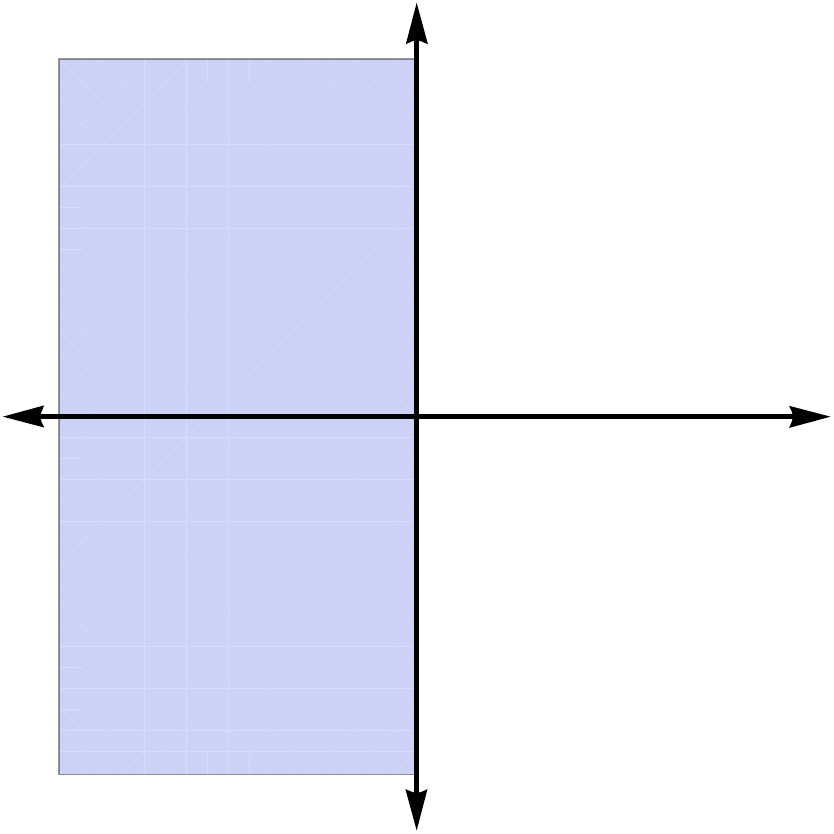}
 & \rule{0ex}{9em}\includegraphics[scale=.35]{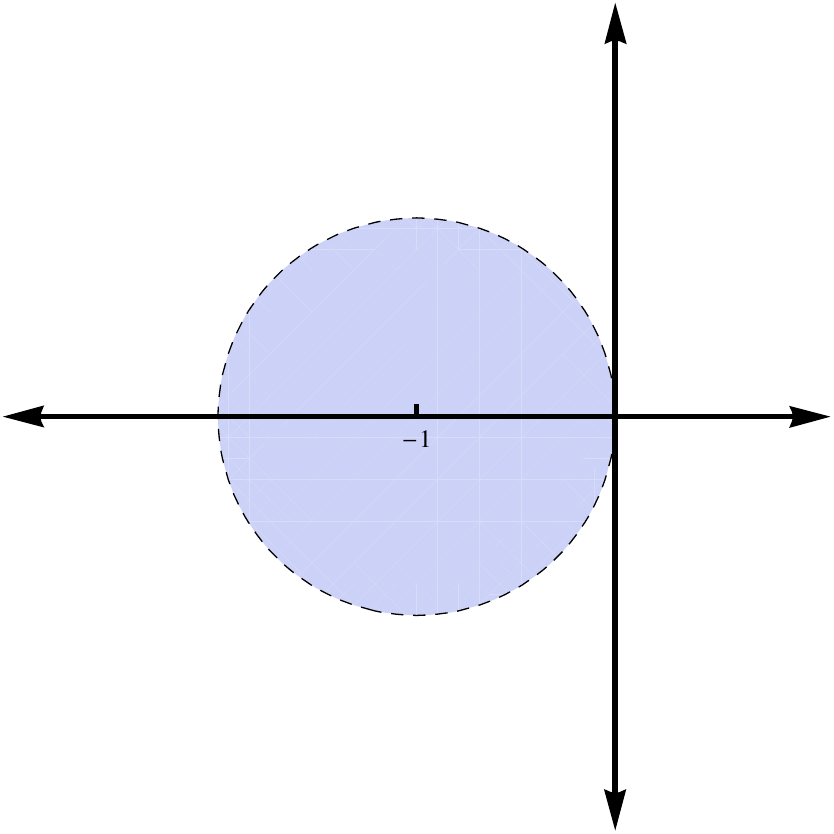} & \rule{0ex}{9em}\includegraphics[scale=.35]{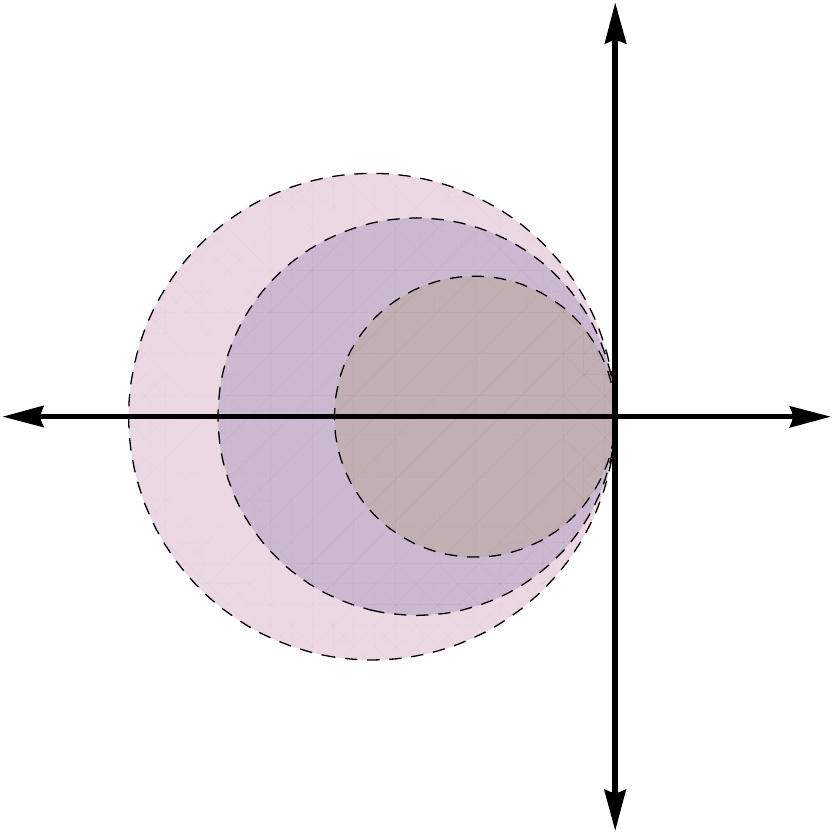}\\ \hline
\end{tabular}
\end{table}

A {\it time scale} $\T$ is any nonempty, (topologically) closed subset of the real numbers $\R$. Thus time scales can be (but are not limited to) any of the usual integer subsets (e.g. $\Z$ or $\N$), the entire real line $\R$, or any combination of discrete points unioned with closed intervals. For example, if $q>1$ is fixed, the {\it quantum time scale} $\overline{q^\Z}$ is defined as
    $$
    \overline{q^\Z}:=\{q^k:k\in\Z\}\cup\{0\}.
    $$
The quantum time scale appears throughout the mathematical physics literature, where the dynamical systems of interest are the $q$-difference equations \cite{Bo,Ca,ChKa}. Another interesting example is the {\it pulse time scale} $\mathbb{P}_{a,b}$ formed by a union of closed intervals each of length $a$ and gap $b$:
    $$
    \mathbb{P}_{a,b}:=\bigcup_k \left[k(a+b),k(a+b)+a\right].
    $$
This time scale is used to study duty cycles of various waveforms. Other examples of interesting time scales include any collection of discrete points sampled from a probability distribution, any sequence of partial sums from a series with positive terms, or even the infamous Cantor set.

The bulk of engineering systems theory to date rests on two time scales, $\R$ and $\Z$ (or more generally $h\Z$, meaning discrete points separated by distance $h$). However, there are occasions when necessity or convenience dictates the use of an alternate time scale. The question of how to approach the study of dynamical systems on time scales then becomes relevant, and in fact the majority of research on time scales so far has focused on expanding and generalizing the vast suite of tools available to the differential and difference equation theorist. We now briefly outline the portions of the time scales theory that are needed for this paper to be as self-contained as is practically possible.

\subsection{The Time Scales Calculus} We now review the time scales calculus needed for the remainder of the paper.

The {\it forward jump operator} is given by $\s(t):=\inf_{s\in\T}\{s>t\}$, while the {\it backward jump operator} is $\rho(t):=\sup_{s\in\T}\{s<t\}$. The {\it graininess function} $\m(t)$ is given by $\m(t):=\s(t)-t$.


A point $t\in\T$ is {\it right-scattered\/} if $\s(t)>t$ and {\it right dense\/} if $\s(t)=t$. A point $t\in\T$ is {\it left-scattered\/} if $\rho(t)<t$ and {\it left dense\/} if $\rho(t)=t$. If $t$ is both left-scattered and right-scattered, we say $t$ is {\it isolated} or {\it discrete}. If $t$ is both left-dense and right-dense, we say $t$ is {\it dense}. The set $\T^\kappa$ is defined as follows: if $\T$ has a left-scattered maximum $m$, then $\T^\kappa=\T-\{m\}$; otherwise, $\T^\kappa=\T$.

For $f:\T\to\R$ and $t\in\T^\kappa$, define $f^\D(t)$ as the number (when it exists), with the property that, for any $\varepsilon > 0$, there exists a neighborhood $U$ of $t$ such that
	\begin{equation}\label{epsdef}
	\left|[f(\sigma(t))-f(s)]-f^\D(t)[\sigma(t)-s]\right|
	\leq\epsilon|\sigma(t)-s|, \quad \forall s\in U.
	\end{equation}
The function $f^\D:\T^\kappa\to\R$ is called the \textit{delta derivative} or the {\it Hilger derivative} of $f$ on $\T^\kappa$. Equivalently, \eqref{epsdef} can be restated to define the $\Delta$-differential operator as
    $$
    x^\Delta(t):={x(\s(t))-x(t)\over \mu(t)},
    $$
where the quotient is taken in the sense that $\m(t)\to 0^+$ when $\m(t)=0$.

\begin{table}
\caption{Differential operators on time scales.}
\label{derivatives}
\renewcommand{\arraystretch}{2}
\begin{tabular}{|c|c|c|}
  \hline
  \sf{time scale} & \sf{differential operator} & \sf{notes}\\ \hline
  $\T$ & $x^\Delta (t)={x(\s(t))-x(t)\over \m(t)}$ & generalized derivative\\ \hline
  $\R$ & $x^\Delta (t)=\lim_{h\to 0}{x(t+h)-x(t)\over h}$ & standard derivative\\ \hline
  $\Z$ & $x^\Delta(t)=\Delta x(t):=x(t+1)-x(t)$ & forward difference\\ \hline
  $h\Z$ & $x^\Delta(t)=\Delta_h x(t):={x(t+h)-x(t)\over h}$ & $h$-forward difference\\ \hline
  $\overline{q^\Z}$ & $x^\Delta(t)=\Delta_q x(t):={x(qt)-x(t)\over (q-1)t}$ & $q$-difference\\ \hline
  ${\mathbb P}_{a,b}$ & $x^\Delta(t)=\begin{cases} {dx\over dt}, & \s(t)=t,\\ {x(t+b)-x(t)\over b}, &\s(t)>t\end{cases}$ & pulse derivative\\ \hline
\end{tabular}
\end{table}

A benefit of this general approach is that the realms of differential equations and difference equations can now be viewed as but special, particular cases of more general {\it dynamic equations on time scales}, i.e. equations involving the delta derivative(s) of some unknown function. See Table~\ref{derivatives}.

Since the graininess function induces a measure on $\T$, if we consider the Lebesgue integral over $\T$ with respect to the $\mu$-induced measure,
    $$
    \int_{\T}f(t)\,d\mu(t),
    $$
then all of the standard results from measure theory are available \cite{Gu}. In particular, under mild technical assumptions on the integrand, we obtain the set of integral operators in Table~\ref{integrals}.

\begin{table}
\caption{Integral operators on time scales.}
\label{integrals}
\renewcommand{\arraystretch}{2}
\begin{tabular}{|c|c|c|}
  \hline
  \sf{time scale} & \sf{integral operator} & \sf{notes}\\ \hline
  $\T$ & $\int_\T f(t)\Delta t$ & generalized integral\\ \hline
  $\R$ & $\int_a^b f(t)\Delta t =\int_a^b f(t)\,dt$ & standard Lebesgue integral\\ \hline
  $\Z$ & $\int_a^b f(t)\Delta t =\sum_{t=a}^{b-1} f(t)$ & summation operator\\ \hline
  $h\Z$ & $\int_a^b f(t)\Delta t =\sum_{t=a}^{b-h} f(t)h$ & $h$-summation \\ \hline
  $\overline{q^\Z}$ & $\int_a^b f(t)\Delta t =\sum_{t=a}^{b/q} {f(t)\over (q-1)t}$ & $q$-summation\\ \hline
\end{tabular}
\end{table}

The upshot here is that the derivative and integral concepts (and all of the concepts in Table~\ref{comparisons}) apply just as readily to {\it any} closed subset of the real line as they do on $\R$ or $\Z$. Our goal is to leverage this general framework against wide classes of dynamical and control systems. Progress in this direction has been made in transforms theory \cite{DaGrJaMaRa, MaGrDa}, control \cite{DaGrJaMa, GrDaDa, GrDaDaMa}, dynamic programming \cite{SeSaWu}, and biological models \cite{HoJa1, HoJa2}.

\subsection{The Hilger Complex Plane}

For $\m>0$, define the \textit{Hilger complex numbers}, the \textit{Hilger real axis}, the \textit{Hilger alternating axis}, and the \textit{Hilger imaginary axis} by
	\begin{alignat*}{2}
	{\mathbb C}_\m&:=\left\{z\in{\mathbb C}:z\neq -\frac{1}{\m}\right\},
	&\quad
	\R_\m&:=\left\{z\in\R:z> -\frac{1}{\m}\right\},\\
	{\mathbb A}_\m&:=\left\{z\in\R:z< -\frac{1}{\m}\right\},
	&\quad
	{\mathbb I}_\m&:=\left\{z\in{\mathbb C}:\left|z+\frac{1}{\m}\right|=\frac{1}{\m}\right\},
	\end{alignat*}
respectively.  For $\m=0$, let
$\mathbb C_0:=\mathbb C,\:\R_0:=\R,\:{\mathbb A}_0:=\emptyset,$ and $\I_0:=i\R$. See Figure~\ref{fig1}.


For $a,b\in{\mathbb C}_\m$, if we define the binary operation $a\oplus b:=a+b+\m a b$, then $(\mathbb C_\m,\oplus)$ forms an abelian group.

Let $\m>0$ and $z\in\mathbb C_\m$.  The \textit{Hilger real part of z} is defined
by
	$$
	\Rem(z):=\frac{|z\m+1|-1}{\m},
	$$
and the \textit{Hilger imaginary part of z} is defined by
	$$
	\Imm(z):=\frac{\text{Arg}(z\m+1)}{\m},
	$$
where $\text{Arg}(z)$ denotes the principal argument of $z$ (i.e., $-\pi<\text{Arg}(z)\leq\pi$). See Figure~\ref{fig1}.

For $\m>0$, define the strip
	$$
	\Z_\m:=\left\{z\in\mathbb C:-\frac{\pi}{\m}<\text{Im}(z)\leq\frac{\pi}{\m}\right\},
	$$
and for $\m=0$, set $\Z_0:=\mathbb C$. Then the {\it cylinder transformation} $\xi_\m:\mathbb C_\m\to\Z_\m$ is given by
	\begin{equation}\label{cyl}
	\xi_\m(z):=\frac{1}{\m}\Log(1+z\m),
	\end{equation}
where $\Log$ is the principal logarithm function.  When $\m=0$, set $\xi_0(z)=z$, for all $z\in\mathbb C$. Then the {\it inverse cylinder transformation} $\xi_\m^{-1}:\Z_\m\to\mathbb C_\m$ is
	\begin{equation}\label{inv_cyl}
	\xi_\m^{-1}(z):={e^{z\m}-1\over \m}.
	\end{equation}
See Figure~\ref{fig1}.

\begin{figure}
	\centering
	\includegraphics[scale=.55]{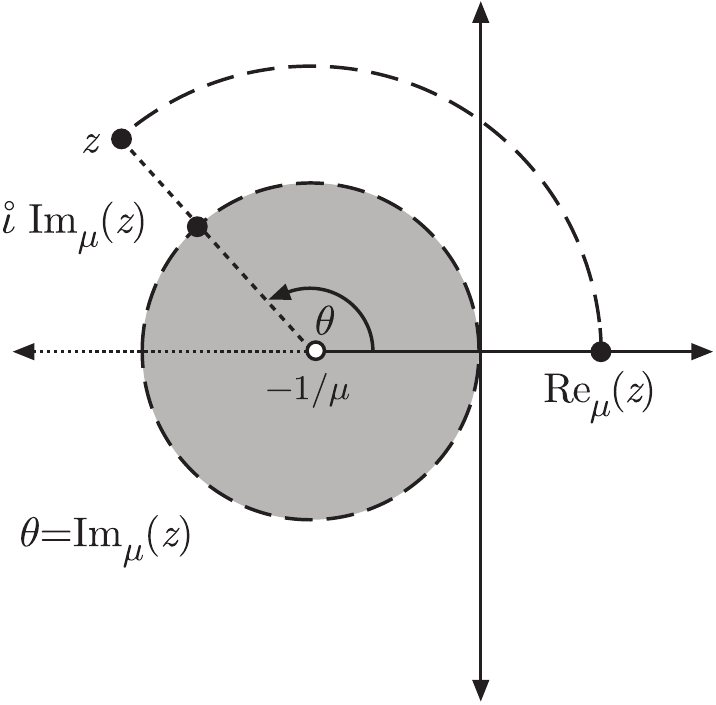}
    \includegraphics[scale=.55]{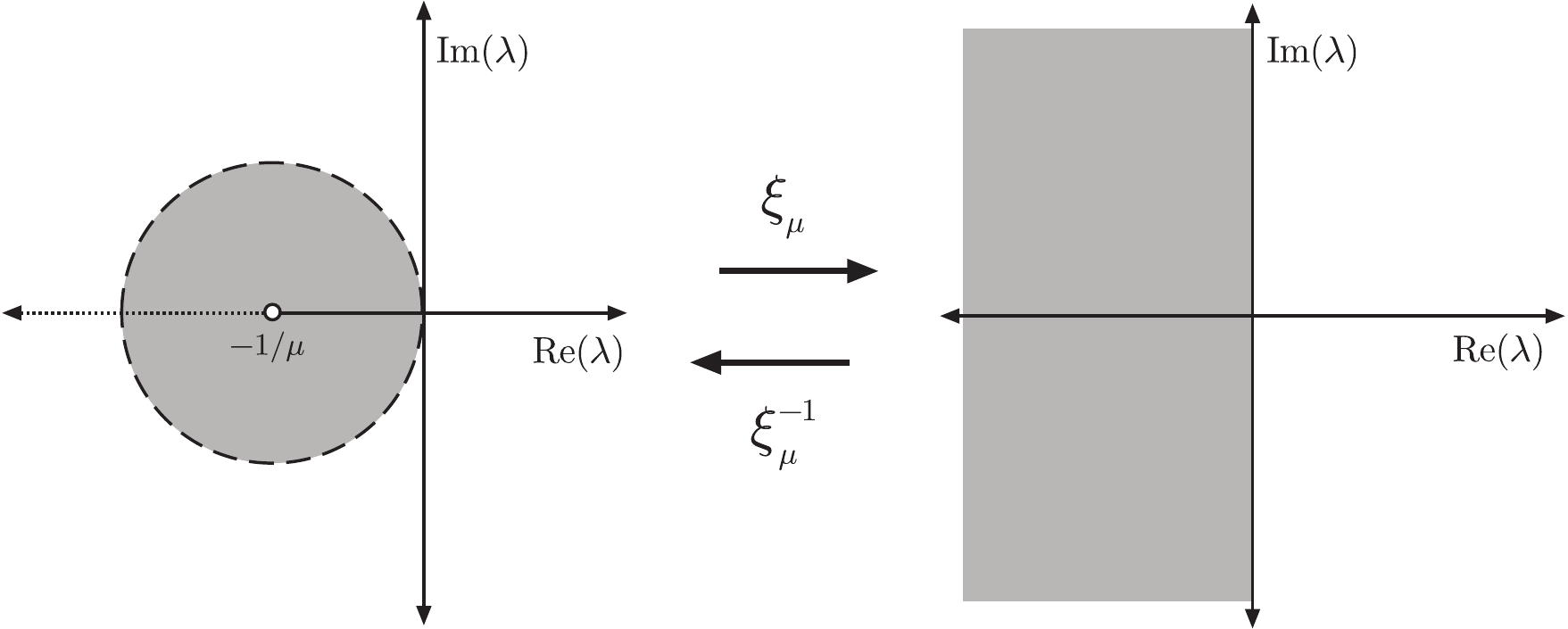}
	\caption{Left: The geometry of the Hilger complex plane. Right: The cylinder \eqref{cyl} and inverse cylinder \eqref{inv_cyl} transformations map the familiar stability region in the continuous case to the Hilger circle in the general time scale case.}
	\label{fig1}
\end{figure}

The region $\Rem z<0$ is naturally important for stability questions involving the linear time invariant system $x^\Delta (t)=Ax(t)$. We call this region the {\it Hilger circle} and denote it by
    $$
    \H:=\H_\m=\{z\in\mathbb C_\m:\Rem(z)<0\}=\{z\in\mathbb C_\m:|1+\m z|<1\}.
    $$
Note that as $\m\to 0^+$, $\H\to\mathbb C^-$, the standard region of exponential stability for the linear time invariant system $\dot{x}=Ax$. On the other hand, as $\m\to 1$, $\H$ becomes the standard region of convergence for the discrete linear time invariant system $\Delta x=Ax$ (shifted one unit to the left due to the difference equation form rather than recursive form of the system). See the bottom row of Table~\ref{comparisons} as well as Figure~\ref{fig1}.

Since the graininess may not be constant for a given time scale, we might interchangeably subscript various quantities (such as $\H$ or $\xi$) with $\m(t)$ instead of $\m$ to emphasize this.


\subsection{Generalized Exponential Functions}

Before we can use the cylinder transformation to define the generalized exponential function on a time scale, we need to define appropriate classes of function spaces on which to work.

A function $p:\T\to\R$ is {\it rd-continuous on $\T$} if $p$ is continuous at right-dense points of $\T$ and has finite left-hand limits at left-dense points of $\T$. We denote this space of functions by $C_\textup{rd}(\T,\R)$. A matrix is rd-continuous provided each of its entries is rd-continuous.

A function $p:\T\to\R$ is \textit{regressive} if $1+\mu(t)p(t)\neq 0$ for all $t\in\T^\kappa$, and this motivates the definition of the following sets: 	
	\begin{gather*}
	\mathcal{R}:=\{p\in C_\textup{rd}(\T,\R) : 1+\m(t)p(t)\not=0\  \forall t\in\T^\kappa\},\\
	\mathcal{R}^+:=\{p\in\mathcal{R}: 1+\mu(t)p(t)>0\ \forall t\in\T^\kappa\}.
	\end{gather*}
A matrix is regressive provided all of its eigenvalues are regressive.

$\mathcal{R}$ under the operation $\oplus$ also forms an abelian group, and the additive inverse of $\oplus$ is denoted by $\ominus$. For $p\in\mathcal R$, $\ominus p\in\mathcal R$ where
    $$
    \ominus p(t):={-p(t)\over 1+\m(t)p(t)}.
    $$

If $p\in\mathcal{R}$, we define the \textit{generalized time
scale exponential function} $e_p(t,s)$ as the unique solution to the dynamic initial value problem
    $$
    x^\Delta(t)=p(t)x(t), \qquad x(s)=1,\qquad t,s\in\T^\kappa.
    $$
It can be shown \cite{BoPe2} that a closed form for $e_p(t,s)$ is given by
	$$
	e_p(t,s)=\text{exp}\(\int^t_s\xi_\m(p(\t))\D\t\)=\exp\left(\int_s^t {\Log(1+\mu(\tau)p(\tau))\over \mu(\tau)}\,\Delta\tau\right),
	$$
where $\xi_\mu$ is the cylinder transformation.

The following theorem is a compilation of properties of $e_p(t,t_0)$ needed later.

\begin{theorem}\label{e_props}
For $p\in\mathcal R$, the function $e_p(t,t_0)$ has the following properties:
\begin{itemize}
    \item[\textup{(i)}] $e_p(t,r)e_p(r,s)=e_p(t,s)$  for all $r,s,t\in\T$.
    \item[\textup{(ii)}] $e_p(t,s)={1\over e_p(s,t)}=e_{\ominus p}(s,t)$
    \item[\textup{(iii)}] $e_p(t,s)e_q(t,s)=e_{p\oplus q}(t,s)$
    \item[\textup{(iv)}] If $q\in\mathcal R$, then ${e_p(t,s)\over e_q(t,s)}=e_{p\ominus q}(t,s)$.
    \item[\textup{(v)}] If $\T=\R$, then
    $e_p(t,s)=\exp\(\int^t_sp(\t)d\t\)$.  If $p$ is
    constant, then $e_p(t,s)=e^{p(t-s)}$.
    \item[\textup{(vi)}]If $\T=\Z$, then
    $e_p(t,s)=\prod^{t-1}_{\t=s}(1+p(\t))$. Moreover, if $\T=h\Z$, with $h>0$ and $p$ is
    constant, then $e_p(t,s)=(1+hp)^{\frac{t-s}{h}}$.
\end{itemize}
\end{theorem}

%
%

\section{Linear State Feedback}

Let $A\in \R^{n\times n}$, $B\in \R^{n\times m}$, and $C\in \R^{p\times n}$ be rd-continuous on $\T$ with $p,m\le n$, and consider the open-loop state equation
    \begin{align*}
    x^{\Delta}(t)&=A(t)x(t)+B(t)u(t), \quad x(t_0)=x_0,\\
    y(t)&=C(t)x(t).
    \end{align*}
In the presence of a linear state feedback controller, we replace the input $u(t)$ above with $u(t):=K(t)x(t)+N(t)r(t)$,
where $r(t)$ represents a new input signal, and $K(t) \in \R^{m \times n}$, $N(t) \in R^{m \times m}$ are rd-continuous. The corresponding closed-loop system is
    \begin{align*}
    x^{\Delta}(t)&=\left[A(t)+B(t)K(t)\right]x(t)+B(t)N(t)r(t),\quad x(t_0)=x_0, \\
    y(t)&=C(t)x(t).
    \end{align*}
Without loss of generality, we proceed with $r(t)\equiv 0$.

\begin{definition}\cite{DaGrJaMa}
Let $A(t)\in\R^{n \times n}, B(t)\in \R^{n \times m}, C(t) \in \R^{p \times n},$ and $D(t)\in \R^{p \times m}$ all be rd-continuous functions on $\T$, with $p,m \leq n$.  The regressive linear system
    \begin{equation}\label{csystem}
    \begin{aligned}
    x^{\Delta}(t) &=A(t)x(t)+B(t)u(t), \quad x(t_0)=x_0,\\
    y(t) &=C(t)x(t) + D(t)u(t),
    \end{aligned}
    \end{equation}
is {\em controllable on $[t_0,t_f]$} if given any initial
state $x_0$ there exists a rd-continuous input signal $u(t)$ such
that the corresponding solution of the system satisfies
$x(t_f)=x_f$.
\end{definition}

Our first result establishes that a necessary and sufficient condition for controllability of the linear system \eqref{csystem} is the invertibility of an associated Gramian matrix.

\begin{theorem}\cite{DaGrJaMa}
\label{controlgramian}
The regressive linear system
    \begin{align*}
    x^{\Delta}(t) &= A(t)x(t)+B(t)u(t), \quad x(t_0)=x_0,\\
    y(t) &= C(t)x(t) + D(t)u(t),
    \end{align*}
is controllable on $[t_0,t_f]$ if and only if the $n\times n$ controllability Gramian matrix given by
    \begin{equation*}\label{g}
    \mathscr{G}_{C}(t_0,t_f):=\int_{t_0}^{t_f}\Phi_{A}(t_0,\sigma(t))B(t)B^{T}(t)\Phi^{T}_{A}(t_0,\sigma(t))\,\Delta t,
    \end{equation*}
is invertible, where $\Phi_Z(t,t_0)$ is the transition matrix for the system $X^\Delta(t) =Z(t)X(t)$, $X(t_0)=I$.
\end{theorem}

Note that as  $\m\to 0^+$, this time scale version of the controllability Gramian matches the standard controllability Gramian on $\R$, and as $\m\to 1$ it matches the standard controllability Gramian on $\Z$ (modulo the unit shift).

\begin{lemma}\label{higcircclosed}
The Hilger circle $\mathcal H$ is closed under the operation $\oplus$ for
all $t \in \T$.
\end{lemma}

\begin{proof}
Let $\alpha \in {\mathbb C}$ be such that $|\alpha|<1$.  Then for a given
graininess $\mu$, the number $a=\frac{\alpha-1}{\mu} \in \H$.
Similarly, let $\beta \in {\mathbb C}$ be such that $|\beta|<1$, so that
$b=\frac{\beta-1}{\mu}\in \H$.  We set $$c:=a \oplus b =a+b+\mu
ab.$$Now, $c \in \H$ if there exists a $\gamma \in {\mathbb C}$ such that
$|\gamma|<1$ with $c=\frac{\gamma-1}{\mu}$.  We claim that the
choice $\gamma=\alpha \beta$ will suffice, from which the conclusion
follows immediately.

Indeed, with this choice of $\gamma$,
    $$
    \frac{\gamma-1}{\mu}=\frac{\alpha-1}{\mu}+\frac{\beta-1}{\mu}
    +\mu\frac{\alpha-1}{\mu}\frac{\beta-1}{\mu},
    $$
and since $|\gamma|=|\alpha| \cdot |\beta| <1$, the claim follows.
\end{proof}

Unlike on $\R$ or $\Z$, in the general time scales setting, there are various ways one could legitimately define {\it exponential stability}. P\"otzsche, Siegmund, and Wirth \cite{PSW} first did so by bounding the state vector above by a decaying {\it regular} exponential function. DaCunha \cite{Da3} generalized their definition by allowing the state vector to be bounded above by a {\it time scale} exponential function of the form $e_{-\lambda}(t,t_0)$. Even more recently, various authors \cite{Li,PeRa} have used a time scale exponential of the form $e_{\ominus \lambda}(t,t_0)$.

In this paper, we adopt DaCunha's definition, but these other definitions could also be used by modifying our arguments slightly.

\begin{definition} \cite{Da3}
The regressive linear state equation
    \begin{align*}
    x^{\Delta}(t) & =A(t)x(t)+B(t)u(t), \qquad x(t_0)=x_0, \\
    y(t) & =C(t)x(t),
    \end{align*}
is {\em uniformly exponentially stable with rate
$\lambda>0$}, where $-\lambda \in \mathcal{R}^+$,
if there exists a constant $\gamma>0$ such that for any $t_0\in\T$ and
$x_0$ the corresponding solution satisfies
    $$
    \|x(t)\| \leq \gamma e_{-\lambda}(t,t_0) \|x_0\|, \qquad t \geq t_0.
    $$
\end{definition}

\begin{lemma}[{\bf Stability Under State Variable Change}]
\label{expstablem}
The regressive linear state equation
    \begin{align*}
    x^{\Delta}(t) & =  A(t)x(t)+B(t)u(t), \qquad x(t_0)=x_0, \\
    y(t) & =  C(t)x(t),
    \end{align*}
is uniformly exponentially
stable with rate $\frac{\lambda +\alpha}{1+\mu_{\max}\alpha}$, where
$\lambda,\alpha>0$ such that $-\lambda \in
\mathcal{R}^+$, if the linear state equation
    $$
    z^{\Delta}(t)=[A(t)(1+\mu\alpha)+\alpha I]z(t),
    $$
is uniformly exponentially stable with rate $\lambda$.
\end{lemma}

\begin{proof}
By direct calculation, $x(t)$ satisfies
    $$
    x^{\Delta}(t)=A(t)x(t), \qquad x(t_0)=x_0,
    $$
if and only if $z(t)=e_{\alpha}(t,t_0)x(t)$ satisfies
    \begin{align}
    \label{varchange}
    z^{\Delta}(t)=[A(t)(1+\mu \alpha)+\alpha I]z(t),\qquad z(t_0)=x_0.
    \end{align}
Now assume there exists a $\gamma>0$ such that for any $x_0$ and
$t_0$ the solution of \eqref{varchange} satisfies $$\|z(t)\| \leq
\gamma e_{-\lambda}(t,t_0)\|x_0\|, \, \, \, t \geq t_0.$$Then,
substituting for $z(t)$ yields
    $$
    \|e_{\alpha}(t,t_0)x(t)\|=e_{\alpha}(t,t_0)\|x(t)\| \leq
\gamma e_{-\lambda}(t,t_0)\|x_0\|,$$ so that $$\|x(t)\| \leq
\gamma e_{-\lambda \ominus \alpha}(t,t_0) \leq \gamma
e_{-(\lambda+\alpha)/(1+\mu_{\max}\alpha)}(t,t_0).
    $$
An application of Lemma~\ref{higcircclosed} then gives the result.
\end{proof}

\begin{theorem}[{\cite[Thm.~1.23]{Ja}; cf.~\cite[Thm.~3.2]{Da3}}]
\label{stabilitycrit}
Suppose $A(t)\in \mathcal{R}(\T,\R^{n \times n})$. The regressive time varying linear dynamic system
    $$
    x^{\Delta}(t)=A(t)x(t), \quad x(t_0)=x_0,
    $$
is uniformly exponentially
stable if there exists a symmetric matrix $Q(t)\in C_{{\rm
rd}}^{1}(\T,\R^{n \times n})$ such that for all $t \in \T$
\begin{enumerate}
\item[(i)]$\eta I \leq Q(t) \leq \rho I$,
\item[(ii)] $\left[(I+\mu(t)A^T(t))Q(\sigma(t))(I+\mu(t)A(t))-Q(t)\right] / \mu(t) \leq -\nu I$,
\end{enumerate}
where $\nu,\eta, \rho>0$ and $-\frac{\nu}{\rho}\in \mathcal{R}^+$.
\end{theorem}

In order to achieve the desired stabilization result, we need to
define a weighted version of the controllability Gramian.
For $\alpha>0$ define the {\it $\alpha$-weighted controllability Gramian matrix} $\G_{C_\alpha}(t_0,t_f)$ by
    $$
    \G_{C_\alpha}(t_0,t_f) :=
    \int_{t_0}^{t_f}(e_{\alpha}(t_0,s))^4\Phi_{A}(t_0,\sigma(s))B(s)B^T(s)\Phi_{A}^T(t_0,\sigma(s))\Delta
    s.
    $$

We are now in position to prove the main result of the paper.

\begin{theorem}[{\bf Gramian Exponential Stability Criterion}]
\label{gramexpstab}
Consider the regressive linear state equation
    \begin{align*}
    x^{\Delta}(t) & =  A(t)x+B(t)u(t), \qquad x(t_0)=x_0, \\
    y(t) & =  C(t)x(t),
    \end{align*}
on a time scale $\T$ such that $\mu_{\min}\le \mu(t)\le\mu_{\max}$ for all $t\in\T$. Suppose there exist constants
$\varepsilon_1,\varepsilon_2>0$ and a strictly increasing function
$\C:\T \to \T$ such that $0<\C(t)-t \leq M$ holds for some constant $0<M<\infty$ and all $t \in \T$ with
    \begin{equation}
    \label{bound}
    \varepsilon_1 I \leq \G_{C}(t,\C(t)) \leq
    \varepsilon_2 I, \quad\text{for all }t\in\T.
    \end{equation}
Then given $\alpha>0$, the state feedback gain
    \begin{equation}
    \label{gain}
    K(t):=-B^T(t)(I+\mu(t)A^T(t))^{-1}\G_{C_\alpha}^{-1}(t,\C(t)),
    \end{equation}
has the property that the resulting closed-loop state equation is uniformly
exponentially stable with rate $\alpha$. We call $\C(t)$ the {\em controllability window} for the problem.
\end{theorem}

\begin{proof}
We first note that for $N=\displaystyle \sup_{t \in \T}\frac{\log(1+\mu(t)\alpha)}{\mu(t)}$, we have $0<N<\infty$ since $\T$ has bounded graininess.  Thus,
    \begin{align*}
    e_{\alpha}(t,\C(t)) & =  \exp\left(-\int_{t}^{\C(t)}\frac{\log(1+\mu(s)\alpha)}{\mu(s)}\Delta s \right) \\
    & \geq   \exp\left(-\int_{t}^{\C(t)}N\Delta s \right) \\
    & =  e^{-N(\C(t)-t)} \\
    & \geq  e^{-MN}.
    \end{align*}
Comparing the quadratic forms $x^T\G_{C_\alpha}(t,\C(t))x$
and $x^T\G_{C}(t,\C(t))x$
gives
    $$
    e^{-4MN}\G_{C}(t,\C(t)) \leq \G_{C_\alpha}(t,\C(t)) \leq
\G_{C}(t,\C(t)), \qquad \text{for all }t\in\T.
    $$
Thus, \eqref{bound} implies
    \begin{equation}
    \label{boundweight}
    \varepsilon_1e^{-4MN}I \leq\G_{C_\alpha}(t,\C(t)) \leq \varepsilon_2 I,
    \text{for all }t\in\T,
    \end{equation}
and so the existence of
$\G_{C_\alpha}^{-1}(t,\C(t))$ is immediate.  Now, we show
that the linear state equation
    \begin{equation*}
    \label{weightvarchange}
    z^{\Delta}(t)=[\hat{A}(t)(1+\mu(t)\alpha)+\alpha I]z(t),
    \end{equation*}
where
    $$
    \hat{A}(t)=A(t)-B(t)B^T(t)(I+\mu(t)A^T(t))\G_{C_\alpha}^{-1}(t,\C(t)),
    $$
is uniformly exponentially stable by applying
Theorem~\ref{stabilitycrit} with the choice
    \begin{equation*}
    \label{qmatrix}
    Q(t)=\G_{C_\alpha}^{-1}(t,\C(t)).
    \end{equation*}
Lemma~\ref{expstablem} then gives the desired result.  To apply the theorem, we first note that $Q(t)$ is symmetric and continuously differentiable.  Thus, \eqref{boundweight} implies
    \begin{equation*}
    \label{qbound}
    \frac{1}{\varepsilon_2}I \leq Q(t) \leq \frac{e^{4MN}}{\varepsilon_1} I,\qquad\text{for all }t\in\T.
    \end{equation*}
Hence, it only remains to show that there exists $\nu>0$ such that
    $$
    \frac{\left[I+\mu(t)\left[(1+\mu(t)\alpha)\hat{A}(t)+\alpha I\right]^T\right]Q^\sigma(t)\left[I+\mu(t)\left[(1+\mu(t)\alpha)\hat{A}(t)+\alpha I \right]\right]-Q(t)}{\mu(t)}\le -\nu I.
    $$

We begin with the first term, writing
    \begin{align*}
    &\left[I+\mu(t)\left[(1+\mu(t)\alpha)\hat{A}(t)+\alpha I\right]^T\right]Q(\sigma(t))\left[I+\mu(t)\left[(1+\mu(t)\alpha)\hat{A}(t)+\alpha I \right]\right] \\
    & \qquad= (1+\mu(t)\alpha)^2\left[\left[I+\mu(t)A^T(t)\right]-\G_{C_\alpha}^{-1}(t,\C(t))\left[I+\mu(t)A(t)\right]^{-1}\mu(t)B(t)B^T(t)\right]\\ & \qquad\qquad \cdot  \G_{C_\alpha}^{-1}(\sigma(t),\C(\sigma(t)))
     \left[\left[I +\mu(t)A(t)\right]-\mu(t)B(t)B^T(t)\left[I+\mu(t)A^T(t)\right]^{-1}\G_{C_\alpha}^{-1}(t,\C(t))\right].
    \end{align*}
We pause to establish an important identity.  Notice that
    \begin{equation}\label{identity1}
    \left[I+\mu(t)A(t)\right]\G_{C_\alpha}(t,\C(t))\left[I+\mu(t)A^T(t)\right] =\mu(t)B(t)B^T(t)+\frac{\G_{C_\alpha}(\sigma(t),\C(t))}{(1+\mu(t)\alpha)^4}.
    \end{equation}
This leads to
    \begin{equation}\label{identity2}
    \begin{aligned}
    &I-\left[I+\mu(t)A(t)\right]^{-1}\mu(t)B(t)B^T(t)\left[I+\mu(t)A^T(t)\right]^{-1}
    \G_{C_\alpha}^{-1}(t,\C(t)) \\
    &=(1+\mu(t)\alpha)^{-4}\left[I+\mu(t)A(t)\right]^{-1}
    \G_{C_\alpha}(\sigma(t),\C(t))\left[I+\mu(t)A^T(t)\right]^{-1}\G_{C_\alpha}^{-1}(t,\C(t)),
    \end{aligned}
    \end{equation}
which in turn yields
    \begin{equation}\label{identity3}
    \begin{aligned}
    & I-\G_{C_\alpha}^{-1}(t,\C(t))\left[I+\mu(t)A(t)\right]^{-1}\mu(t)B(t)B^T(t)
    \left[I+\mu(t)A^T(t)\right]^{-1} \\
    &=(1+\mu(t)\alpha)^{-4}\G_{C_\alpha}^{-1}(t,\C(t))\left[I +\mu(t)A(t)\right]^{-1}\G_{C_\alpha}(\sigma(t),\C(t))\left[I+\mu(t)A^T(t)\right]^{-1}.
    \end{aligned}
    \end{equation}
The first term can now be rewritten as
    \begin{align*}
    &(1+\mu(t)\alpha)^2\left[\left[I+\mu(t)A^T(t)\right]-\G_{C_\alpha}^{-1}(t,\C(t))
    \left[I+\mu(t)A(t)\right]^{-1}\mu(t)B(t)B^T(t)\right]\\
    &\qquad\cdot\G_{C_\alpha}^{-1}(\sigma(t),\C(\sigma(t)))\left[\left[I +\mu(t)A(t)\right]-\mu(t)B(t)B^T(t)\left[I+\mu(t)A^T(t)\right]^{-1}\G_{C_\alpha}^{-1}(t,\C(t))\right]\\
    &=(1+\mu(t)\alpha)^2\left[I -\G_{C_\alpha}^{-1}(t,\C(t))\left[I+\mu(t)A(t)\right]^{-1}\mu(t)B(t)B^T(t)
    \left[I+\mu(t)A^T(t)\right]^{-1}\right]\\
    &\qquad\cdot\left[I+\mu(t)A^T(t)\right]\G_{C_\alpha}^{-1}(\sigma(t),\C(t))\left[I+\mu(t)A(t)\right]\\
    &\qquad\cdot\left[I-\left[I+\mu(t)A(t)\right]^{-1}\mu(t)B(t)B^T(t)\left[I+\mu(t)A^T(t)\right]^{-1}
    \G_{C_\alpha}^{-1}(t,\C(t))\right]
    \end{align*}
Using \eqref{identity2} and \eqref{identity3}, we can now write
    \begin{equation}\label{firstterm}
    \begin{aligned}
    &\left[I+\mu(t)\left[(1+\mu(t)\alpha)\hat{A}^T(t)+\alpha I\right]\right]Q(\sigma(t))\left[I+\mu(t)\left[(1+\mu(t)\alpha)\hat{A}(t)+\alpha I \right] \right]\\
    &=(1+\mu(t)\alpha)^{-6}\G_{C_\alpha}^{-1}(t,\C(t))\left[I+\mu(t)A(t)\right]^{-1}
    \G_{C_\alpha}(\sigma(t),\C(t))\G_{C_\alpha}^{-1}(\sigma(t),\C(\sigma(t)))\\
    &\qquad\cdot\G_{C_\alpha}(\sigma(t),\C(t))\left[I+\mu(t)A^T(t)\right]^{-1}\G_{C_\alpha}^{-1}(t,\C(t)).
    \end{aligned}
    \end{equation}
On the other hand, from the definition of $\G_{C_\alpha}(t,\C(t))$, we have
    $$
    \G_{C_\alpha}(\sigma(t),\C(\sigma(t))) \geq \G_{C_\alpha}(\sigma(t),\C(t)),
    $$
which in turn implies
    $$
    \G_{C_\alpha}^{-1}(\sigma(t),\C(\sigma(t))) \leq \G_{C_\alpha}^{-1}(\sigma(t),\C(t)).
    $$
Combining this with \eqref{firstterm} gives
    \begin{align*}
    &\left[I+\mu(t)\left[(1+\mu(t)\alpha)\hat{A}^T(t)+\alpha I\right]\right]Q(\sigma(t))\left[I+\mu(t)\left[(1+\mu(t)\alpha)\hat{A}(t)+\alpha I \right] \right]\\
    &\leq (1+\mu(t)\alpha)^{-6}\G_{C_\alpha}^{-1}(t,\C(t))\left[\left[I+\mu(t)A(t)\right]^{-1}
    \G_{C_\alpha}(\sigma(t),\C(t))\left[I+\mu(t)A^T(t)\right]^{-1}\right]\G_{C_\alpha}^{-1}(t,\C(t)).
    \end{align*}
Applying \eqref{identity1} again yields
    \begin{align*}
    &\left[I+\mu(t)\left[(1+\mu(t)\alpha)\hat{A}^T(t)+\alpha I\right]\right]Q(\sigma(t))\left[I+\mu(t)\left[(1+\mu(t)\alpha)\hat{A}(t)+\alpha I \right] \right] \\
    &\leq(1+\mu(t)\alpha)^{-6}\G_{C_\alpha}^{-1}(t,\C(t))\\
    &\qquad\cdot\left[(1+\mu(t)\alpha)^4 \G_{C_\alpha}(t,\C(t))-(1+\mu(t)\alpha)^4\left[I+\mu(t)A(t)\right]^{-1}
    \mu(t)B(t)B^T(t)\left[I+\mu(t)A^T(t)\right]^{-1}\right]\\
    &\qquad\cdot\G_{C_\alpha}^{-1}(t,\C(t))\\
    &\leq(1+\mu(t)\alpha)^{-2}\G_{C_\alpha}^{-1}(t,\C(t)).
    \end{align*}
Thus,
    \begin{align*}
    &\frac{\left[I+\mu(t)\left[(1+\mu(t)\alpha)\hat{A}^T(t)+\alpha I\right]\right]Q(\sigma(t))\left[I+\mu(t)\left[(1+\mu(t)\alpha)\hat{A}(t)+\alpha I \right] \right]-Q(t)}{\mu(t)} \\
    &\qquad\leq -\frac{(1+\mu(t)\alpha)^2-1}{\mu(t)(1+\mu(t)\alpha)^2}\G_{C_\alpha}^{-1}(t,\C(t)) \\
    &\qquad\leq -\frac{(1+\mu(t)\alpha)^2-1}{\mu(t)\varepsilon_2(1+\mu(t)\alpha)^2}I.
    \end{align*}
This last quantity is not necessarily constant, but since the graininess of $\T$ has a (presumably nonzero) upper bound,
    $$
    \frac{(1+\mu(t)\alpha)^2-1}{\mu(t)\varepsilon_2(1+\mu(t)\alpha)^2}
    =\frac{2\alpha+\mu(t)\alpha^2}{\varepsilon_2(1+\mu(t)\alpha)^2}\geq \frac{\alpha}{\varepsilon_2(1+\mu_{\max}\alpha)^2}.
    $$
Setting $\nu:=\alpha^{-1}\varepsilon_2(1+\mu_{\max}\alpha)^2$, we obtain
    $$
    \frac{\left[I+\mu(t)\left[(1+\mu(t)\alpha)\hat{A}^T(t)+\alpha I\right]\right]Q(\sigma(t))\left[I+\mu(t)\left[(1+\mu(t)\alpha)\hat{A}(t)+\alpha I \right] \right]-Q(t)}{\mu(t)}\le -\nu I.
    $$
\end{proof}

Several natural questions arise regarding Theorem~\ref{gramexpstab}:
    \begin{itemize}
    \item[(Q1)] Is the assumption \eqref{bound} for $\C(t)$ reasonable?
    \item[(Q2)] Does $\C(t)$ have a single, unified form regardless of the time scale?
    \item[(Q3)] How does any such unified $\C(t)$ compare to its analogue on $\R$ and $\Z$?
    \item[(Q4)] How does the state feedback gain in \eqref{gain} compare to its analogue on $\R$ and $\Z$?
    \end{itemize}

First, the assumed bound \eqref{bound} is not severe since it is a reformulation of the controllability Gramian invertibility criterion in Theorem~\ref{controlgramian}, and controllability of the open-loop system is a prerequisite for feedback stabilization. The requirement that $\C(t)$ be increasing on an interval just ensures a nondegenerate interval (in the time scale) on which the open-loop system is controllable.

In response to (Q2), for any $\delta_1,\delta_2>0$, we propose the controllability window
    \begin{equation*}\label{generalC}
    \C(t):=
    \begin{cases}
    t+\delta_1,             & \text{if }\s(t)=t,\\
    \sigma^k(t),            & \text{if $\sigma^i(t)\neq t$ for all $0 \leq i \leq k$},\\
    \sigma^k(t)+\delta_2,   & \text{else},
    \end{cases}
    \end{equation*}
where $\s^k$ means the composition of the forward jump operator $\s$ with itself $k-1$ times.

Note that, for $\T=\R$, $\C(t)=t+\delta$ for any $\delta>0$ is sufficient, while on $\T=\Z$, the function $\C(t)=t+k$ for $k \in \N$ meets the criteria. These coincide with the controllability windows found in the literature for both the continuous and discrete cases \cite{AnMi,CaDe,Ru}, giving a very satisfying answer to (Q3).

Finally, we remark that the general form of the state feedback gain in \eqref{gain} coalesces nicely with the known forms of $K(t)$. When $\T=\R$, \eqref{gain} takes the form
    $$
    K(t)=-B^T(t)\G_{C_\alpha}^{-1}(t,t+\delta),
    $$
where $\C(t)=t+\delta$, $\delta>0$, and it is shown in \cite{Ch,Ru} that the system is stabilized. On the other hand, when $\T=\Z$, \eqref{gain} becomes
    \begin{equation}\label{zgain}
    K(t)=-B^T(t)A^{-T}\G_{C_\alpha}^{-1}(t,t+k),
    \end{equation}
where $\C(t)=t+k$, $k\in\N$. Note that \eqref{zgain} is a shifted version (again, due to the difference equation formulation rather than recursive formulation of the problem) of the familiar discrete state feedback gain \cite{Ru}.

\section{Experimental Results}

Throughout the preceding discussion, it has been assumed that the time scale
is known \textit{a priori}; in other words, that a system's time domain is
known before the system dynamics \textquotedblleft start\textquotedblright\
at time time $t=0$. Under this assumption, it is possible to calculate
feedback gain $K(t)$ \textit{a priori} if the system's state matrices $A(t)$
and $B(t)$ are known. Scenarios in which non-standard time scales (not $\mathbb{R}$ or $h%
\mathbb{Z}
$)\ are useful may come about for different reasons. For example, it may be
that a computer controller cannot guarantee consistent hard deadlines (i.e.
``real time" response) for communication with sensors and actuators; in this
case, a time scale may be scheduled that is more amenable to the other tasks
the computer is performing. A similar problem may occur in a networked, or
distributed, control system, in which various network traffic activities
determine the time scale. In either case, if the time scale is known, or at
least known over some finite window into the future, the feedback gain may
be computed and applied in advance.

To illustrate the paper's central theorem in hardware, a simple experiment
was devised using a DC motor with an intertial mass. A system
identification procedure produced approximate 2$^{nd}$-order state matrices%
\begin{eqnarray*}
\hat{A} &=&\left[
\begin{array}{cc}
0 & 1 \\
0 & -0.15%
\end{array}%
\right] ,\qquad \hat{B}=\left[
\begin{array}{c}
0 \\
13.8%
\end{array}%
\right] , \\
\frac{d\hat{x}(t)}{dt} &=&\hat{A}\hat{x}(t)+\hat{B}\hat{u}(t),\qquad t\in
\mathbb{R}
,
\end{eqnarray*}%
\newline
where state vector $\hat{x}(t)$ is the motor's angular shaft position (rev)
and velocity (rev/s), and $\hat{u}(t)$ is the input voltage (V). Electrical dynamics were neglected due to the relatively small electrical time constant. The hat
notation designates $\hat{A}$ and $\hat{B}$ as the state matrices of a
dynamical system on $%
\mathbb{R}
$. Sample-and-hold discretization to an arbitrary time scale $\mathbb{T}$
gives
\[
A(t)=\left[ \frac{e^{\hat{A}\mu (t)}-I}{\mu (t)}\right] ,\qquad B(t)=\left[
\sum_{i=1}^{\infty }\frac{(\hat{A}\mu (t))^{i-1}}{i!}\right] \hat{B}%
,\qquad t\in \mathbb{T}.
\]%
Now equation (3.1a) is in force.

\begin{figure}
\includegraphics[scale=.75]{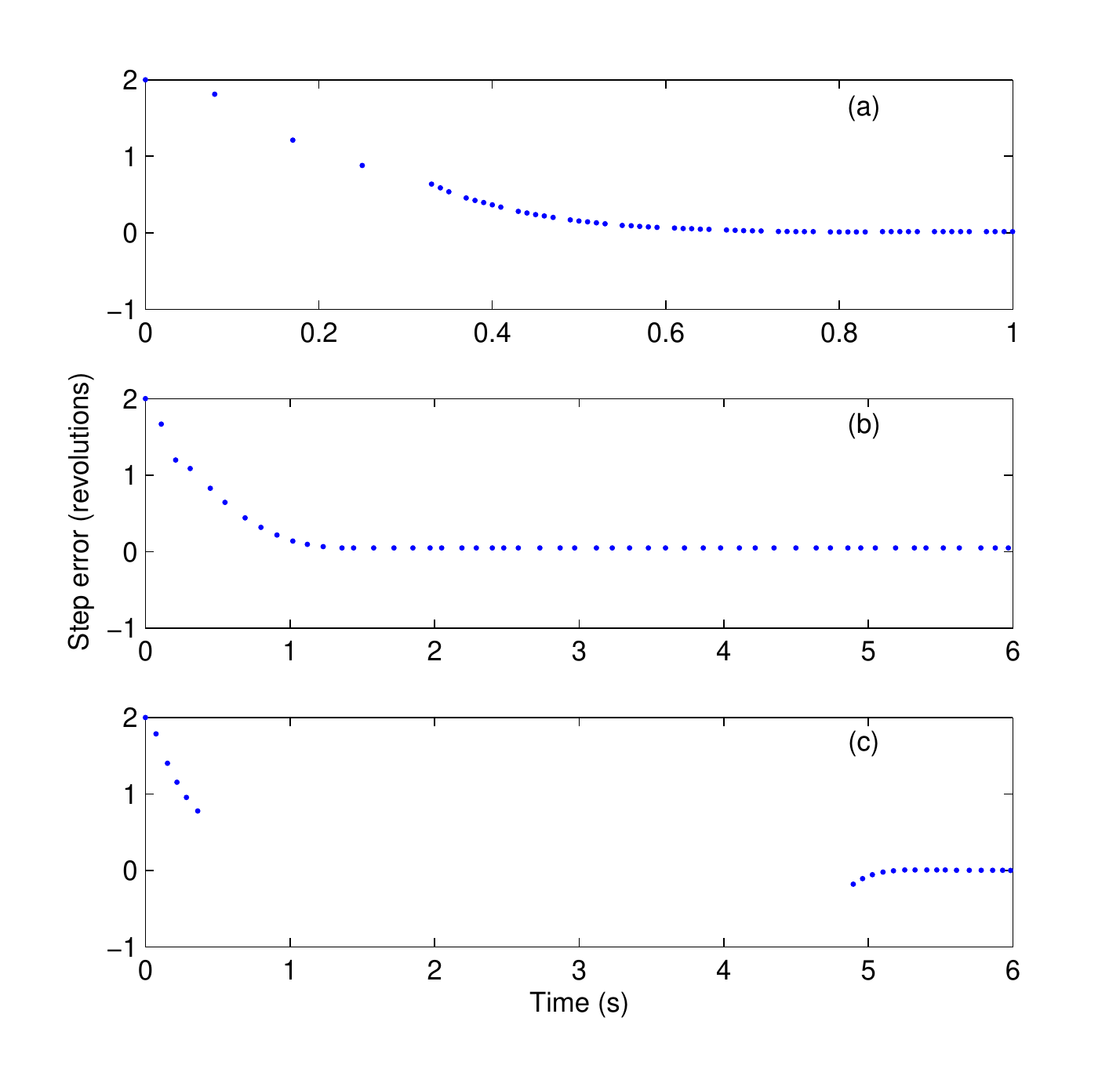}
\caption{Step responses for the three cases discussed in the text.
Note that the time axis in case (a) has been magnified in order to see the
individual points in the time scale. }
\label{FIG_Step_Responses}
\end{figure}

To begin, several discrete time scales $\mathbb{T}$ were selected and
populated with anywhere from $\ell=$ 20 to 100 points (all of the time scales
in these experiments were purely discrete with no continuous intervals).
Choosing a window operator $\mathcal{C}(t)=\sigma ^{k}(t)$ simply amounted
to choosing a window sized $k>0$. Some ramifications of this choice are
discussed later. Next, using MATLAB, $K(t)$ was computed over the first $\ell-k$
points in the time scale. $K(t)$ and $\mathbb{T}$, along with control law $u(t)=K(t)x(t)+N(t)r(t)$  where $N(t)\equiv 1$ and $r(t)=2h(t)$ where $h$ denotes a unit step function and $x(0)\neq 0$ were programmed into a computer running
the QNX operating system and outfitted with digital and analog input/output
hardware. Internal high-precision timers were employed so that the system
would acquire the motor states $x(t)$, and apply drive current $u(t)$, only
at the pre-determined points in $t\in \mathbb{T}$. The resulting state
trajectories therefore illustrate the closed-loop system step response.

Three examples of the closed-loop step response are shown in Figure~\ref{FIG_Step_Responses}. Time scale $\mathbb{T}_{a}$ of example (a) was created
with widely varying graininess. Graininess $\mu _{a}(t)$ occurs in multiples
of 10ms, with the first four points exhibiting graininess of 80 or 90ms, and
points thereafter exhibiting graininess of 10 or 20ms. This time scale was
designed to emulate the timing of a real-time process that is unable to meet
hard 10ms deadlines. If a deadline is missed, i.e. the controller cannot
respond at the next specified $t\in \mathbb{T}_{a}$, the next point in the
time scale is scheduled some multiple of 10ms in the future. Time scale $%
\mathbb{T}_{b}$ of example (b) exhibits graininess from a uniformly random
distribution between 80 and 150ms. The third example (c) combines two
interesting phenomena: a time scale $\mathbb{T}_{c}$ of uniformly random
distribution, with a very large gap in the middle. Example (c) is
particularly interesting because it can be seen that the controller has
computed its best estimate (as close as the model allows) of the open-loop
constant input current required to move the motor shaft to near-zero error
by the end of the gap.

\begin{figure}
\includegraphics[scale=.75]{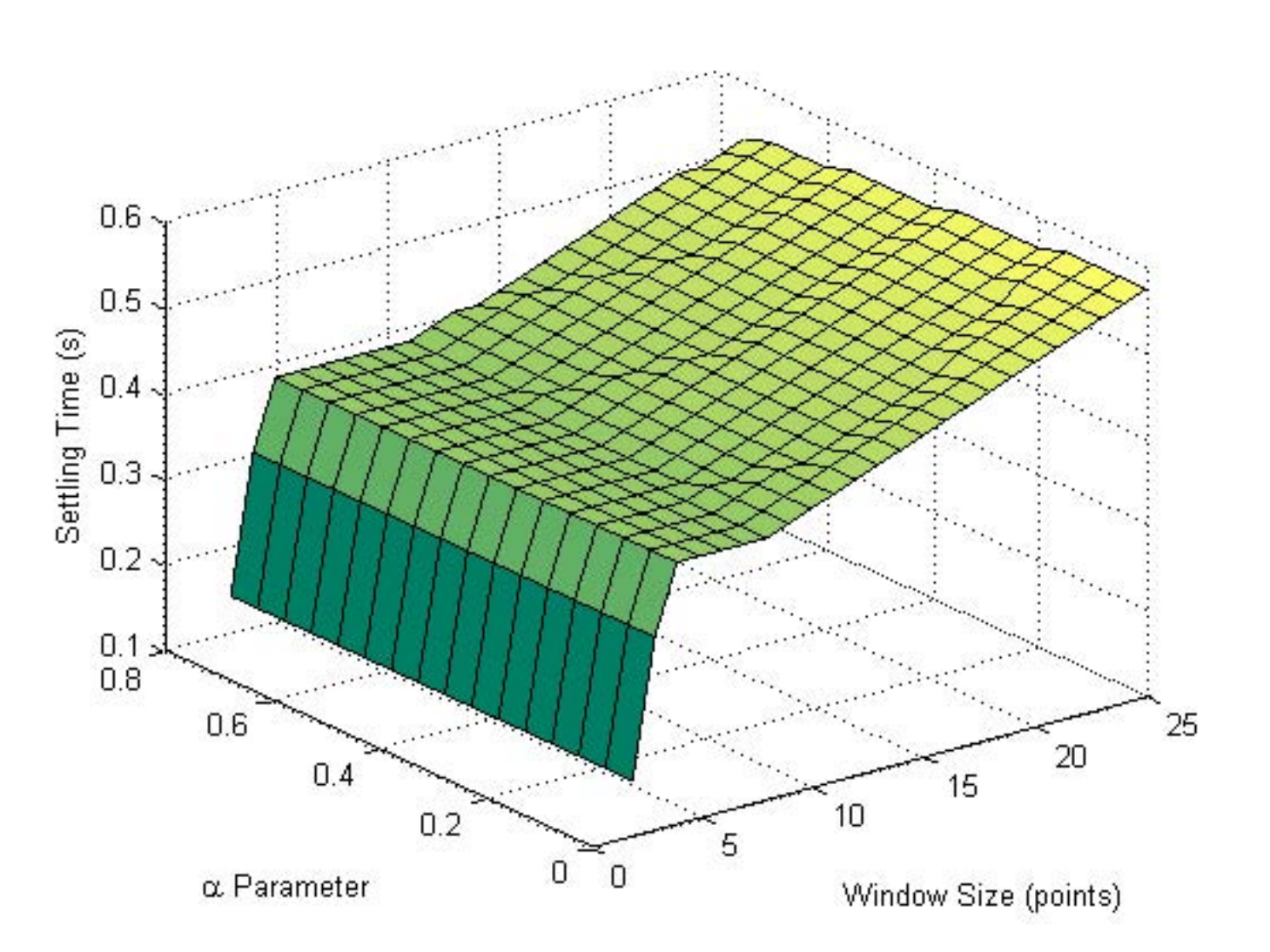}
\caption{An image plotting parameter $\protect\alpha $ versus
window size $k$ versus settling time for experiment (a). The surface is
rough because of the discretization of the time scale.}
\label{FIG_BSB_T_SETTLE}
\end{figure}

\begin{figure}
\includegraphics[scale=.75]{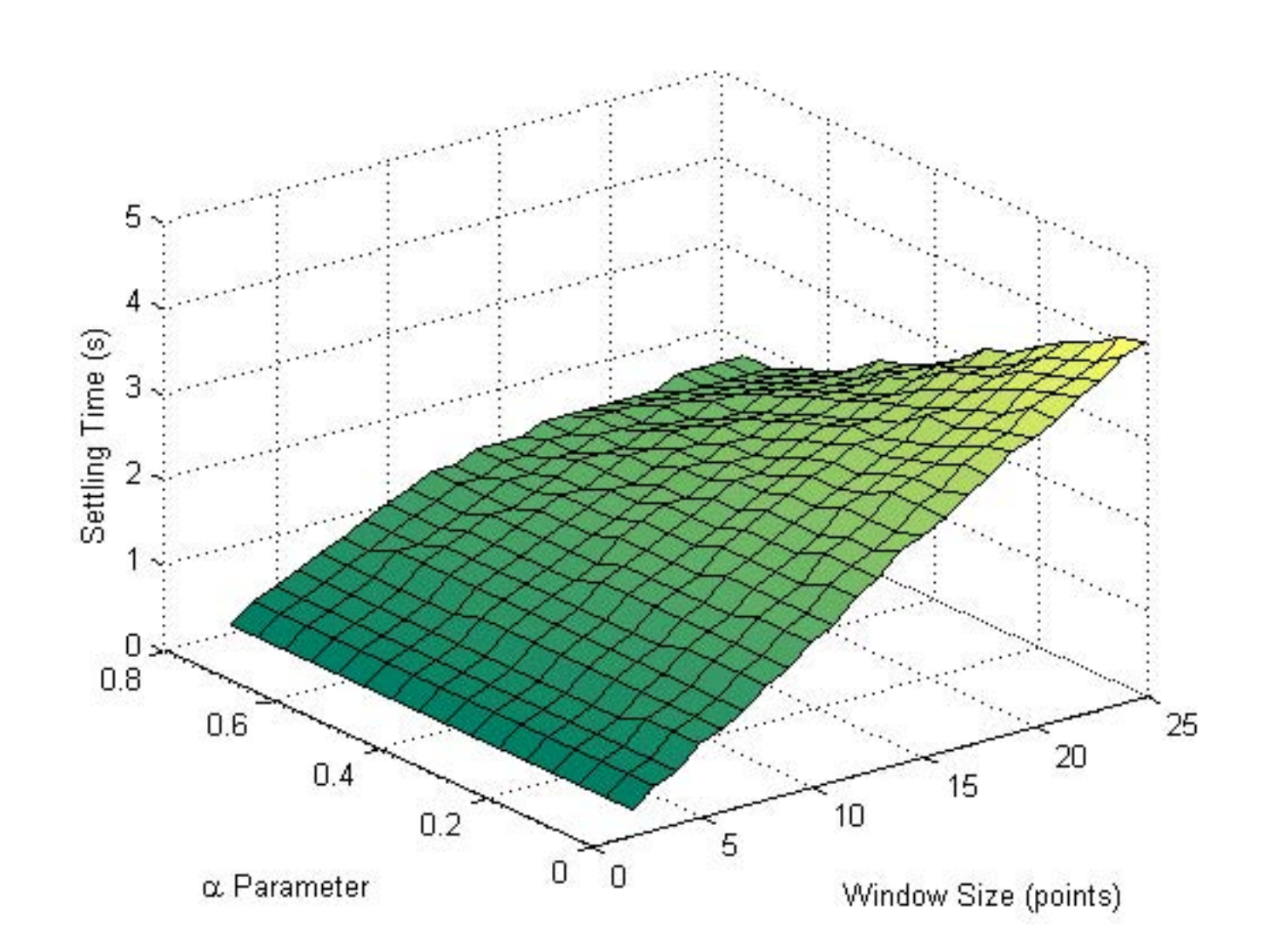}
\caption{An image plotting parameter $%
\protect\alpha $ versus window size $k$ versus settling time for experiment
(b).}
\label{FIG_RANDOM_T_SETTLE}
\end{figure}

In each example, it was necessary to choose a window size $k$ as well as a
constant $\alpha $ for the computation of $K(t)$. The performance impact of
different choices is not obvious, and is explored through simulation in
Figures~\ref{FIG_BSB_T_SETTLE} and \ref{FIG_RANDOM_T_SETTLE}. These figures show
how step response settling times (the time required for the response to fall
within 10\% of its final value) are dependent on $k$ and $\alpha $, for $%
k\in \{2,...,25\}$ and $\alpha \in \lbrack 0.01,0.79]$. Figure~\ref{FIG_BSB_T_SETTLE} shows that the settling time of example (a) is relatively
insensitive to changes in $\alpha $, but highly sensitive to changes in
window size $k$. Small windows would seem to produce better performance;
however they also induce large magnitudes in the values of \thinspace $K$,
which in turn produce large magnitudes in $u$ that may exceed the physical
limitations of the system. Figure~\ref{FIG_RANDOM_T_SETTLE}\ shows the same
analysis for example (b); one difference is that now the settling time is
somewhat more sensitive to the choice of $\alpha $.

These examples illustrate that the full-state, closed-loop feedback $%
u(t)=K(t)x(t)$ will indeed stabilize a simple 2nd-order system on a variety
of interesting time scales. However, there are still several practical
limitations to overcome. First, the actual computation of $K(t)$ is very
complex, and it is doubtful that small embedded control processors could
compute $K(t)$ in real time. Second, $K(t)$ depends on knowledge of the time
scale over some finite future window (defined by the operator $\mathcal{C}(t)
$). Thus, $K$ is not strictly causal (although it does not depend on
knowledge of the system states in the future). Lastly, $K$ depends on
knowledge of the system parameters, which are often not well known. It
should be noted, however, that the first and third of these limitations also
apply in the ``classical" cases (feedback control on $%
\mathbb{R}
$ and $%
\mathbb{Z}
$). The second limitation is obviated on $%
\mathbb{R}
$ and $%
\mathbb{Z}
$ because the time scale is always known \textit{a priori}.


\end{document}